\documentclass[11pt,reqno]{amsart}
\usepackage[utf8]{inputenc}
\usepackage[T1]{fontenc}
\usepackage{lmodern}
\usepackage{amsfonts,amsthm,amsmath,amssymb,thmtools}
\usepackage{graphicx}
\usepackage[dvipsnames,table]{xcolor}   % »table« optons loads »colortbl« package
\usepackage{enumerate}
\usepackage{hyperref}
\usepackage{color}
\usepackage{tikz-cd}
\usepackage{transparent}
\usepackage{fvextra}
\usepackage[margin=1in]{geometry}
\usepackage[
    maxbibnames=99,
    backend=biber,
    style=alphabetic,
    sorting=nyt,
    giveninits=true
]{biblatex}
\DeclareFieldFormat{pages}{#1}
\renewbibmacro{in:}{%
  \ifentrytype{article}
    {}
    {\bibstring{in}%
     \printunit{\intitlepunct}}}
\DeclareFieldFormat
[article,inbook,incollection,inproceedings,patent,thesis,unpublished]
  {title}{\mkbibemph{#1}}
\DeclareFieldFormat{journaltitle}{#1\isdot}
\DeclareFieldFormat[article]{volume}{\mkbibbold{#1}}
\DeclareFieldFormat[article]{number}{\bibstring{number}\addnbspace #1}

\renewbibmacro*{journal+issuetitle}{
  \usebibmacro{journal}
  \setunit*{\addspace}
  \iffieldundef{series}
    {}
    {\newunit
     \printfield{series}
     \setunit{\addspace}}
  \printfield{volume}
  \setunit{\addspace}
  \usebibmacro{issue+date}
  \setunit{\addcomma\space}
  \printfield{number}
  \setunit{\addcolon\space}
  \usebibmacro{issue}
  \setunit{\addcomma\space}
  \printfield{eid}
  \newunit}

\newtheoremstyle{mystyle}%                % Name
  {}%                                     % Space above
  {0}%                                     % Space below
  {\itshape}%                                     % Body font
  {}%                                     % Indent amount
  {\bfseries}%                            % Theorem head font
  {.}%                                    % Punctuation after theorem head
  { }%                                    % Space after theorem head, ' ', or \newline
  {\thmname{#1}\thmnumber{ #2}\thmnote{ (#3)}}%                                     % Theorem head spec (can be left empty, meaning `normal')

\theoremstyle{mystyle}
\newtheorem{Thm}{Theorem}[section]

\newtheorem{Prop}[Thm]{Proposition}
\newtheorem{Conj}[Thm]{Conjecture}

\theoremstyle{definition}

\newtheorem{Cons}[Thm]{Construction}

\theoremstyle{remark}

\declaretheoremstyle[%
  spaceabove=3pt,%reduce or increase between theorem and proof
  spacebelow=10pt,%reduce or increase
  headfont=\normalfont\itshape,%
  postheadspace=.5em,%
  qed=\qedsymbol%
]{mystyle2}

\newcommand{\Z}{\mathbb{Z}}
\newcommand{\Q}{\mathbb{Q}}

\author{Qiuyu Ren}
\address{Department of Mathematics, University of California, Berkeley, Berkeley, CA 94720, USA}
\email{qiuyu\_ren@berkeley.edu}
\title{Families of cosmetic surgeries}

\begin{document}

\begin{abstract}
We construct infinite families of chirally cosmetic surgeries on chiral hyperbolic knots and purely cosmetic surgeries on hyperbolic manifolds with multiple cusps, disproving conjectures that these phenomena do not appear, including Problem 1.12(d) in the K3 problem list. We also give some hints regarding why chirally cosmetic surgeries appear to be more common than purely cosmetic surgeries on $1$-cusped manifolds.
\end{abstract}

\maketitle

\section{Introduction}

Let $M$ be a compact oriented $3$-manifold with a torus boundary. Two different Dehn fillings on $M$ are \textit{purely cosmetic} if they are orientation-preservingly homeomorphic, and \textit{chirally cosmetic} if they are orientation-reversingly homeomorphic. If $M$ admits an orientation-preserving homeomorphism sending a slope $s$ to a slope $s'$ on $\partial M$, then the Dehn fillings $M(s)$ and $M(s')$ are purely cosmetic. The cosmetic surgery conjecture, proposed by Gordon \cite{gordon1990dehn}, posits that the converse is true.
\begin{Conj}[{\cite[Conjecture~6.2]{gordon1990dehn}}]\label{conj:cos}
If $M\ne S^1\times D^2$ is an irreducible compact oriented $3$-manifold with a torus boundary and $s,s'$ are slopes on $\partial M$ so that $M(s)$ and $M(s')$ are purely cosmetic, then $M$ admits an orientation-preserving homeomorphism sending $s$ to $s'$.
\end{Conj}

The study of the cosmetic surgery conjecture has a long history, and several variants of the conjecture have been studied, including the following. We refer readers to \cite{futer2025excluding} and the K3 list \cite[Problem~1.12]{baykur2026k3} for recent surveys.

(1) Restricting to complements of knots in $S^3$, the cosmetic surgery conjecture states that no nontrivial knots in $S^3$ admit purely cosmetic surgeries. This is arguably the most studied case, where the combined work of Ni--Wu \cite{ni2015cosmetic}, Hanselman \cite{hanselman2022heegaard}, Daemi--Eismeier--Lidman \cite{daemi2024filtered} showed, among other constraints, that if the $s$- and $s'$-surgeries on a nontrivial knot $K$ are purely cosmetic, then $K$ has genus $2$, Alexander polynomial $1$, and $\{s,s'\}=\{2,-2\}$. Futer--Purcell--Schleimer \cite{futer2025excluding} verified the conjecture for knots up to $19$ crossings. Ren \cite{ren2025cosmetic} reduced the conjecture to hyperbolic surgeries on hyperbolic knots.

(2) One may restrict the conjecture to cusped hyperbolic manifolds. In this setting, the following conjecture was proposed, relaxing the condition that $M$ has only one cusp.
\begin{Conj}[{\cite[Conjecture~1.3]{futer2025excluding}}]\label{conj:multi}
Let $M$ be a cusped hyperbolic manifold with at least one cusp. If $\mathbf s,\mathbf{s'}$ are two sets of slopes on the cusps so that the fillings $M(\mathbf s)$ and $M(\mathbf{s'})$ are purely cosmetic, then $M$ admits an orientation-preserving symmetry sending $\mathbf s$ to $\mathbf{s'}$.
\end{Conj}

Here, as in the $1$-cusped case, \textit{cosmetic} still means homeomorphic.

(3) One can study a version of Conjecture~\ref{conj:cos} for chirally cosmetic surgeries, although the obvious modification does not work. In addition to plenty of examples on Seifert fibered spaces, \cite{bleiler1999cosmetic} found chirally cosmetic fillings on a $1$-cusped hyperbolic manifold with inequivalent slopes, \cite{futer2025excluding} found more such pairs, and \cite{ichihara2018cosmetic} found such a pair with the fillings being hyperbolic. Nevertheless, when restricting to knot complements, the following conjecture was proposed.

\begin{Conj}[{\cite[Question~1]{ito2021note}\cite[Conjecture~1]{ichihara2023constraints}}]\label{conj:chiral}
If $K$ is a chiral knot that is not a torus knot $T(2,n)$, then $K$ does not admit chirally cosmetic surgeries.
\end{Conj}

This conjecture appears in the K3 problem list \cite{baykur2026k3} as Problem 1.12(d). It has been verified for some infinite families of knots (see e.g. \cite{ito2026non} and references therein), as well as hyperbolic knots up to $15$ crossings \cite{futer2025excluding}.

\medskip

In this paper, we present a simple construction of infinite families of cosmetic surgeries. The construction is flexible enough to disprove Conjecture~\ref{conj:multi} if the multi-slopes $s,s'$ are allowed to have $\infty$ components (meaning some cusps are allowed to be unfilled), as well as Conjecture~\ref{conj:chiral} and its weaker hyperbolic version stated in \cite[Conjecture~1.6]{futer2025excluding}.

\begin{Thm}\label{thm:multi}
For any $n\ge3$, there exists an $n$-cusped hyperbolic manifold $M$ and slopes $s,s'$ on the first cusp, such that the Dehn fillings $M(s,\infty,\cdots,\infty)$ and $M(s',\infty,\cdots,\infty)$ are hyperbolic and purely cosmetic, and that no homeomorphism of $M$ sends one set of slopes to the other.
\end{Thm}

\begin{Thm}\label{thm:chiral}
There exist infinitely many asymmetric (chiral) knots in $S^3$ with chirally cosmetic hyperbolic surgeries.
\end{Thm}

In Section~\ref{sec:general}, we present a general construction of (pure or chiral) cosmetic surgeries. We include a discussion of why the construction fails to yield purely cosmetic surgeries on $1$-cusped hyperbolic manifolds, the setting most closely related to Conjecture~\ref{conj:cos}. In Section~\ref{sec:multi} and Section~\ref{sec:chiral}, we prove Theorem~\ref{thm:multi} and Theorem~\ref{thm:chiral}, respectively.

\section*{Acknowledgments}
We thank Ian Agol for his continued guidance and support. This work is motivated by computer experiments carried out by the author on the $1$-cusped hyperbolic $10$- and $11$-tetrahedra census manifolds \cite{li2025complete,li_snappy_11_tets_2026}, using the code accompanying \cite{futer2025excluding}. In these (non-exhaustive) experiments, many more chirally cosmetic pairs than in \cite{futer2025excluding} were found, but no purely cosmetic pairs were found. We thank the authors of \cite{futer2025excluding} for making their code available. We thank Nathan Dunfield and Shana Li for their invaluable help with the software SnapPy \cite{SnapPy}. We thank David Futer, Marc Kegel, Daniel Ruberman, Saul Schleimer, and Jonathan Zung for helpful discussions.

\section{The general construction}\label{sec:general}
For simplicity, we work exclusively with oriented hyperbolic $3$-manifolds, although the general construction below works for all compact $3$-manifolds with toroidal boundaries. A \textit{slope} on a torus is either an isotopy class of unoriented simple closed curves, or $\infty$. We follow Thurston's convention that the $\infty$-filling of a cusp corresponds to leaving the cusp unfilled. By Mostow rigidity \cite{mostow1968quasi}, every homeomorphism between hyperbolic $3$-manifolds is homotopic to an isometry. We shall use \textit{symmetry} to refer to a self-isometry of a hyperbolic $3$-manifold, and \textit{symmetry group} to refer to its isometry group, which is isomorphic to its mapping class group (\cite{waldhausen1968irreducible,gabai2003homotopy}).

\begin{Cons}\label{cons:main}\mbox{}\smallskip\\
\textbf{Input}:\vspace{-4pt}
\begin{itemize}
\item A cusped hyperbolic manifold $N$ with $\partial N=T_1\sqcup\cdots T_{k+\ell}$, $k\ge2$, $\ell\ge0$;
\item A symmetry $\phi\colon N\to N$ permuting the first $k$ cusps, sending $T_i$ to $T_{(i\bmod k)+1}$, $i=1,\cdots,k$;
\item A slope $s$ on $T_1$ not preserved under $\phi^k|_{T_1}$.
\end{itemize}
\textbf{Output}:
The manifold $$M:=N(\infty,\phi(s),\phi^2(s),\cdots,\phi^{k-1}(s),\infty,\cdots,\infty)$$ obtained by filling the second to the $k$-th cusps of $N$ by slopes $\phi(s),\cdots,\phi^{k-1}(s)$.
\end{Cons}

The output manifold $M$ from Construction~\ref{cons:main} admits cosmetic fillings $M(s,\infty,\cdots,\infty)$ and $M(\phi^k(s),\allowbreak\infty,\cdots,\infty)$, as $\phi$ extends to a homeomorphism between them. The fillings $M(s,\infty,\cdots,\infty)$ and $M(\phi^k(s),\infty,\cdots,\infty)$ are purely or chirally cosmetic, depending on whether $\phi$ is orientation-preserving or orientation-reversing. When the slope $s$ is long enough, Thurston's hyperbolic Dehn filling theorem implies that $M$ is a hyperbolic manifold, with the cores of the fillings isotopic to the first $k-1$ shortest geodesics of $M$, and the rest of the length spectrum of $M$ within any prescribed interval $[0,L]$ is a small perturbation of that of $N$. It can often be arranged so that $M$ has no symmetry sending $s$ to $\phi^k(s)$, or in fact no nontrivial symmetry at all (in which case $M$ is said to be asymmetric).

Using this model construction, one can easily construct infinite families of chirally cosmetic surgeries on asymmetric $1$-cusped hyperbolic $3$-manifolds. Indeed, we show in Section~\ref{sec:chiral} that such manifolds can even be taken to be knot complements in $S^3$. It is interesting and instructive to see why the same fails for purely cosmetic surgeries.

\begin{Prop}\label{prop:no_pure}
No purely cosmetic surgeries can be constructed for $1$-cusped hyperbolic manifolds via Construction~\ref{cons:main}. That is, it is impossible to supply inputs in the construction with $\ell=0$ and $\phi$ orientation-preserving.
\end{Prop}
\begin{proof}
Assume for contradiction that there is a cusped hyperbolic manifold $N$ with $\partial N=T_1\sqcup\cdots\sqcup T_k$ which admits an orientation-preserving symmetry $\phi$ permuting the cusps cyclically in that order, so that $\phi^k|_{T_1}$ acts nontrivially on the set of slopes on $T_1$. The cusps $T_i$ admit Euclidean structures defined up to global scalings, which are preserved under the isometry $\phi$. Thus, $\phi^k|_{T_1}$ is a finite-order isometry of order $1,2,3,4$ or $6$; orders $1$ or $2$ are impossible by the assumption on the nontrivial action on slopes. The quotient $\bar N:=N/\phi$ is thus a hyperbolic orbifold, whose unique cusp is the Euclidean orbifold $S^2(3,3,3)$, $S^2(2,4,4)$, or $S^2(2,3,6)$, depending on the order of $\phi^k|_{T_1}$ (equivalently, of $\phi^k$). However, $\bar N$ has only arcs and loops of cyclic singularities, so the boundary of the singularity locus cannot consist of exactly three points, a contradiction.
\end{proof}

It is not completely unreasonable to speculate that most cosmetic surgeries do arise from Construction~\ref{cons:main}. If all cosmetic surgeries arise this way, Conjecture~\ref{conj:cos} for hyperbolic manifolds follows from Proposition~\ref{prop:no_pure}.

\section{Purely cosmetic surgeries on multi-cusped hyperbolic manifolds}\label{sec:multi}
We prove Theorem~\ref{thm:multi}. We will apply Construction~\ref{cons:main} to $k=\ell=n-1$.

Let $K$ be a hyperbolic knot in the orbifold $\mathcal O=S^2(2,3,6)\times I$, homotopic in the complement of the singular loci to the meridian of the third singular arc (the arc with order $6$ singularities). We choose $K$ so that its complement has no nontrivial symmetries or hidden symmetries, where a \textit{hidden symmetry} is a symmetry of a finite cover that does not descend to the base. It is standard to arrange for $K$ to be hyperbolic and to have no nontrivial symmetries in a given homotopy class, so we only indicate how to guarantee that $K$ has no hidden symmetries. The point is that by the assumption on the homotopy class, $\pi_1(K)$ normally generates $\pi_1^{orb}(\mathcal O)$, so the argument in \cite[Proof of Lemma~4]{reid1991arithmeticity} shows that if $\mathcal O\backslash K$ has hidden symmetry, then its orientable commensurator quotient (assuming $K$ is non-arithmetic) is a hyperbolic orbifold with only rigid cusps, namely cusps of shapes $S^2(3,3,3)$, $S^2(2,4,4)$, or $S^2(2,3,6)$; see also \cite[Proposition~9.1]{neumann1992arithmetic}. In particular, the cusp field of $\mathcal O\backslash K$ at the cusp $K$ must be $\Q(\sqrt{-1})$ or $\Q(\sqrt{-3})$, which does not happen for a generic $K$. This justifies the choice of $K$.

Let $\mu$ denote a meridian of $K$, and let $\lambda$ denote a longitude of $K$ that is homologous, in the complement of $K$ together with the singular locus, to a meridian of the third singular arc in $\mathcal O$. Pick a slope $r=k\mu+q\lambda$ on $K$ for some large $q$ with $q\equiv5\pmod6$ and perform the $r$-surgery on $K$ to get a hyperbolic orbifold $\mathcal O'=\mathcal O_K(r)$. The first orbifold homology group of $\mathcal O$ is $$H_1^{orb}(\mathcal O)=\Z\{x,y,z\}/(2x,3y,6z,x+y+z)\cong\Z/6,$$ where $x,y,z$ are represented by the meridians along the three singular arcs. The surgery has $$H_1^{orb}(\mathcal O')=\Z\{x,y,z,z_0\}/(2x,3y,6z,x+y+z,z-kz_0)\cong\Z/6k,$$ where $z_0$ is represented by $\mu$. Let $N$ denote the maximal abelian cover of $\mathcal O'$, which is the $6k$-fold cyclic cover associated to the map $\theta\colon\pi_1^{orb}(\mathcal O')\to H_1^{orb}(\mathcal O')\xrightarrow{\cong}\Z/6k$ sending $x,y,z,z_0$ to $3k,2k,k,1$, respectively. Since $\theta$ restricts to the map $\pi_1^{orb}(S^2(2,3,6))\to\Z/6\to\Z/6k$ on each boundary, where the first map gives the $6$-fold regular covering $T^2\to S^2(2,3,6)$ and the second map is the multiplication by $k$, we know that each boundary of $\mathcal O'$ lifts to $k$ cusps in the hyperbolic manifold $N$. Call the lifts $T_1,\cdots,T_k$ and $T_{k+1},\cdots,T_{2k}$. The deck transformation generator $\phi\colon N\to N$ permutes each group of cusps cyclically, and the first return map $\phi^k|_{T_1}$ is an order $6$ rotational symmetry. Thus, we may pick a long slope $s$ on $T_1$, feed the data $(N,\phi,s)$ to Construction~\ref{cons:main}, which yields a hyperbolic manifold $M$ with purely cosmetic surgeries.

The nontrivial symmetry $\phi^{3k}$ extends to a homeomorphism on $M$, fixing the slopes on the cusps. We argue that $M$ can be arranged to have symmetry group exactly $\Z/2$. By first picking the slope $r$ on $K$ long enough, and then the slope $s$ on $T_1$ long enough, we may assume that the first $k-1$ shortest geodesics in $M$ are the cores of the fillings on $N$, and the next shortest geodesic is the lift of the core of the filling on $\mathcal O\backslash K$. Let $N_0\subset N$ denote the lift of $\mathcal O\backslash K$. Then, the systole consideration shows that any symmetry of $M$ comes from a symmetry of $N$, and in fact from a symmetry of $N_0$. Since $N_0$ is a finite cover of $\mathcal O\backslash K$, which is assumed to have no nontrivial symmetries or hidden symmetries, every symmetry of $N_0$ is a deck transformation. The only deck transformations that preserve the cusp $T_2$ together with the slope $\phi(s)$ are the trivial map and $\phi^{3k}|_{N_0}$, hence no other symmetry of $N_0$ extends to $M$. We conclude that $M$ has symmetry group $\Z/2$.

\section{Chirally cosmetic surgeries on chiral knots}\label{sec:chiral}
We prove Theorem~\ref{thm:chiral}. We will apply Construction~\ref{cons:main} to the complement of a $3$-component hyperbolic Brunnian link, with $k=3$, $\ell=0$.

Recall that a \textit{Brunnian link} is a nonsplit link in $S^3$ with at least $3$ components, for which every proper sublink is an unlink. The simplest example of a $3$-component Brunnian link is the Borromean rings. A more complicated $3$-component Brunnian link is shown in Figure~\ref{fig:brun}.

\begin{figure}[htbp]
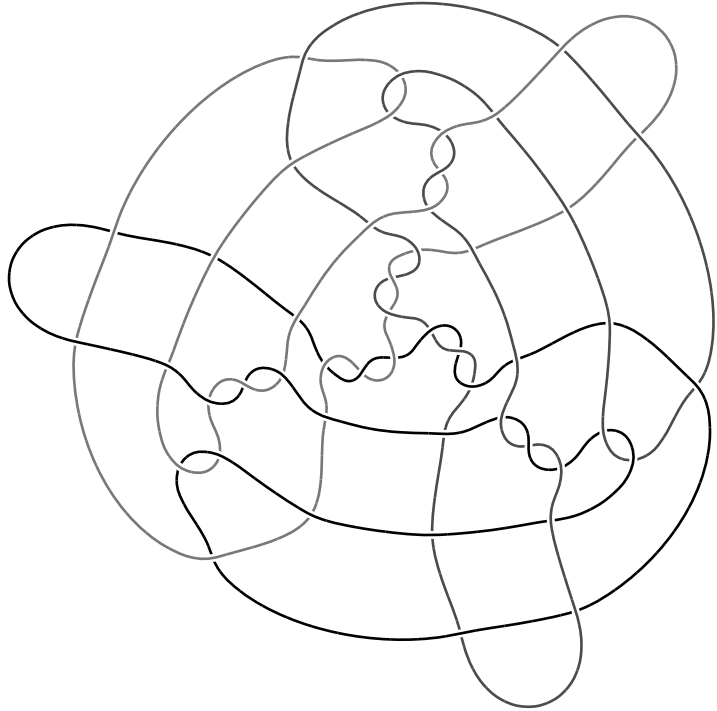

    \centering
    \include{img_L}
    \caption{A $3$-component amphichiral hyperbolic Brunnian link with symmetry group $\Z/6$. The image was created by the program KnotJob \cite{schutz_knotjob_2025}.}
    \label{fig:brun}
\end{figure}

\begin{Prop}\label{prop:brun}
The link $L=L_1\cup L_2\cup L_3$ in Figure~\ref{fig:brun} satisfies the following properties:
\begin{enumerate}
\item It is a $3$-component Brunnian link;
\item It is hyperbolic;
\item The symmetry group of the pair $(S^3,L)$ is $\Z/6$, which is equal to the symmetry group of $N=S^3\backslash L$;
\item A generator $\phi$ of the symmetry group is orientation-reversing on $N$, permuting the components of $L$ cyclically in the forward order.
\end{enumerate}
\end{Prop}

We postpone the proof of Proposition~\ref{prop:brun} and complete the proof of Theorem~\ref{thm:chiral}. 

Since $\phi$ is a symmetry of $(S^3,L)$, it preserves the meridians and longitudes of $L$ up to signs. Therefore, $\phi^3$ negates the slopes on each component of $L$. Let $s=1/q$ be a long slope on $L_1$. Then $\phi(s)$ and $\phi^2(s)$ are $\pm1/q$ slopes on $L_2$ and $L_3$. We feed the data $(N,\phi,s)$ to Construction~\ref{cons:main}, which yields a hyperbolic manifold $M=S^3\backslash K$ with chirally cosmetic surgeries, where the knot $K$ is the image of $L_1$ in the $(\phi(s),\phi^2(s))$ surgery on $L_2\cup L_3\subset S^3$. When $s$ is long enough, the cores of the fillings are the two shortest geodesics in $M$, so any symmetry of $M$ comes from a symmetry of $S^3\backslash L$. By the description of the symmetry group of $L$, any nontrivial symmetry of $S^3\backslash L$ either maps $L_2$ to another component or maps $\phi(s)$ to a different slope on $L_2$, so no nontrivial symmetry of $S^3\backslash L$ extends to $M$. We conclude that $K$ is asymmetric, with chirally cosmetic fillings.

\begin{proof}[Proof of Proposition~\ref{prop:brun}]
(1) is immediate from the diagram.

We locate a $\Z/6$ subgroup of the symmetry group of $(S^3,L)$. There is an obvious $3$-fold rotational symmetry permuting the components cyclically. There is a component-preserving orientation-reversing involution witnessing the fact that $L$ is amphichiral; this can be seen by looking at the skinny disks bounded by each component, regarded as ribbons, and rotating each of them ``outward'' simultaneously by $\pi$, thereby deforming $L$ continuously to its mirror image (obtained by changing every crossing in the diagram). The involution commutes with the $3$-fold rotational symmetry, hence they generate a $\Z/6$ subgroup of the symmetry group of $L$. Moreover, from the description, a generator of this subgroup satisfies the description in (4).

We prove (2) and (3) using verified computation in SnapPy 3.3.2 \cite{SnapPy}, run in SageMath 10.7 with Python 3.13.3.

We create $L$ using its PD code generated by the program KnotFolio \cite{miller_knotfolio}. Then we take its exterior $N$ and verify its hyperbolicity.
\begin{Verbatim}[fontsize=\footnotesize,breaklines=true,breakanywhere=true]
sage: import snappy
sage: L = snappy.Link([(94,60,95,59), (21,48,22,49), (80,58,81,57), (95,64,96,33), (86,18,87,17), (74,28,75,27), (28,76,29,75), (83,12,84,13), (85,26,86,27), (90,40,91,39), (88,8,89,7), (76,14,77,13), (63,32,64,1), (43,18,44,19), (73,14,74,15), (4,42,5,41), (8,88,9,87), (36,70,37,69), (15,72,16,73), (22,62,23,61), (33,96,34,65), (71,16,72,17), (42,4,43,3), (9,70,10,71), (11,82,12,83), (40,6,41,5), (67,54,68,55), (62,24,63,23), (2,52,3,51), (92,38,93,37), (10,78,11,77), (58,66,59,65), (81,34,82,35), (55,66,56,67), (56,80,57,79), (31,46,32,47), (93,52,94,53), (38,92,39,91), (29,84,30,85), (53,68,54,69), (19,50,20,51), (35,78,36,79), (45,30,46,31), (24,48,25,47), (44,26,45,25), (49,20,50,21), (1,60,2,61), (6,90,7,89)])
sage: N = L.exterior()
sage: N.verify_hyperbolicity()[0]
True
\end{Verbatim}
This proves $L$ is hyperbolic. We check the symmetry group of $N=S^3\backslash L$:
\begin{Verbatim}[fontsize=\footnotesize,breaklines=true,breakanywhere=true]
sage: N.symmetry_group()
Z/6
\end{Verbatim}
However, in the current version of SnapPy, the method \texttt{.symmetry\_group()} uses unverified computation. To verify the computation, one needs to compute the canonical triangulation of $N$ by \texttt{iso = N.isometry\_signature(verified=True)} and \texttt{T = snappy.Triangulation(iso,remove\_finite\_\allowbreak vertices=False)}. This computes the Epstein--Penner decomposition of $N$, with non-tetrahedral cells further canonically triangulated; see \cite[Section~3]{fominykh2016census} for more details. Unfortunately, SnapPy fails to compute the isometry signature of $N$, even after a few randomizations, and with the \texttt{exact\_bits\_prec\_and\_degrees} parameter set to \texttt{[(10000,120)]}. Instead, we try to drill out a short geodesic and compute the symmetry group of the drilled manifold.
\begin{Verbatim}[fontsize=\footnotesize,breaklines=true,breakanywhere=true]
sage: N.randomize()
sage: N.length_spectrum_alt(7,bits_prec=400,verified=True)
[Length                                      Core curve  Word
 1.59356570124401... + 2.71947057945369...*I -           x2X18,
 1.59356570124401... + 2.71947057945369...*I -           x2x24x31X18,
 1.59356570124401... + 2.71947057945369...*I -           x24x31,
 1.59356570124401... - 2.71947057945369...*I -           x4x4X10,
 1.59356570124401... - 2.71947057945369...*I -           x1X15,
 1.59356570124401... - 2.71947057945369...*I -           x4X10,
 2.37009806931195... + 0.00000000000000...*I -           x3X28X5x30X17X9X12X16,
 2.37009806931195... + 0.00000000000000...*I -           x8x28x21x22X30X32X17,
 2.37009806931195... + 0.00000000000000...*I -           x25]
sage: N_drill = N.drill_word("x2X18")
sage: iso = N_drill.isometry_signature(verified=True)
sage: iso
'-cqbvLLLLLAMvvLMLALwvwzLwAvPwAvLQLLwQMvzMLwMPQAPAQQPQPQQQceafafajahalamanamawapaxaqaAarauaBaCaKaJazaLaMaHaRaQaQaGa2aMaOaKaSaPa0a9aYaXaVaUa+aZaebZaabeb1a7a8a7aibkb+a5a6a-aibjblbnbdbcbbbgbfbmbfbkbhbpbgbobpbmbnblblbobkbnbpbviiufvalivvawxboauasvoioahhcahpblfnlknvoicxviachhhahigalbcoaoocngkgsgavhtsaonhcao'
sage: T = snappy.Triangulation(iso,remove_finite_vertices=False)
sage: len(T.isomorphisms_to(T))
1
\end{Verbatim}
The method \texttt{.length\_spectrum\_alt(7)} with the \texttt{verified} parameter set to \texttt{True} rigorously computes the complex lengths of the first seven shortest geodesics up to some precision. The output gives nine short geodesics, because the program could not decide which one among the last three is the shortest (from the symmetry, one can expect that they actually have the same length). Similarly, the first six geodesics shown in the output are not guaranteed to be of equal length. Any symmetry of $N$ permutes the first six shortest geodesics, although not necessarily nontrivially a priori. Nevertheless, if the symmetry group of $N$ has cardinality larger than $6$, each of the six shortest geodesics is fixed by some nontrivial symmetry of $N$; in particular, this implies that the manifold $N_{drill}$ we obtained by drilling out the first geodesic shown in the output has a nontrivial symmetry group. But the symmetry group of $N_{drill}$ is the automorphism group of its canonical triangulation $T$ that we construct from its isometry signature, which is shown to have size $1$. This proves the cardinality of the symmetry group of $N$ is at most $6$. But the symmetry group of $N$ contains the symmetry group of $(S^3,L)$, which is shown to contain a $\Z/6$ subgroup. Hence these symmetry groups are both exactly $\Z/6$.
\end{proof}

\printbibliography

\end{document}